\documentclass{amsart}

\usepackage{color,graphicx,amssymb,latexsym,amsfonts,txfonts,amsmath,amsthm}
\usepackage{pdfsync}

\usepackage{hyperref}
\hypersetup{
    colorlinks=true,       
    linkcolor=blue,          
    citecolor=blue,        
    filecolor=blue,      
    urlcolor=blue           
}

\newtheorem{theo}{Theorem}

\newtheorem{coro}{Corollary}
\newtheorem{lemm}{Lemma}

\theoremstyle{remark}
\newtheorem{rema}{\bf Remark}

\begin{document}

\title{Automorphism groups of origami curves}
\author{Rub\'en A. Hidalgo}

\address{Departamento de Matem\'atica y Estad\'{\i}stica, Universidad de La Frontera.  Temuco, Chile}
\email{ruben.hidalgo@ufrontera.cl}

\thanks{Partially supported by Project Fondecyt 1190001}

\maketitle

\begin{abstract}
A closed Riemann surface $S$ (of genus at least one) is called an origami curve if it admits a non-constant holomorphic map $\beta:S \to E$ with at most one branch value, where $E$ is a genus one Riemann surface. In this case, $(S,\beta)$ is called an origami pair and ${\rm Aut}(S,\beta)$ is the group of conformal automorphisms $\phi$ of $S$ such that $\beta=\beta \circ \phi$.  Let $G$ be a finite group. It is a known fact that $G$ can be realized as a subgroup of ${\rm Aut}(S,\beta)$ for a suitable origami pair $(S,\beta)$. It is also known that $G$ can be realized as a group of conformal automorphisms of a Riemann surface $X$ of genus $g \geq 2$ and with quotient orbifold $X/G$ also of genus $\gamma \geq 2$. Given a conformal action of $G$ on a surface $X$ as before, 
we prove that there is an origami pair $(S,\beta)$, where $S$ has genus $g$ and $G \cong {\rm Aut}(S,\beta)$ such that the actions of ${\rm Aut}(S,\beta)$ on $S$ and that of $G$ on $X$ are topologically equivalent. 

\end{abstract}

\section{Introduction}
Let $S$ be a closed Riemann surface of genus $g \geq 1$ and let us denote by ${\rm Aut}(S)$ its group of conformal automorphisms.
We say that $S$ is an {\it origami curve} if it admits a non-constant holomorphic map $\beta:S \to E$ having at most one branch value, where $E$ is a genus one Riemann surface; we say that $\beta$ is an {\it origami map} and that $(S,\beta)$ is an {\it origami pair}. By the Riemann-Hurwitz formula, $\beta$ has no branched values if and only if $g=1$.
We denote by ${\rm Aut}(S,\beta)$ the deck group of $\beta$, that is, the subgroup of ${\rm Aut}(S)$ formed by those automorphisms $h$ such that $\beta=\beta \circ h$. If the degree of $\beta$ is equal to the order of ${\rm Aut}(S,\beta)$, then the origami pair is regular. The origami pair $(S,\beta)$ is called {\it uniform} if all the $\beta$-preimage of its branch value have the same local degree (regular origami pairs are necessarily uniform ones and non-uniform ones must be of genus $g \geq 2$).

Topologically, an origami can  also be described as follows. Consider a finite collection of disjoint unit squares in the plane, with their sides parallels to the coordinate axes. Then consider a gluing of the sides (by translations) such that:
\begin{enumerate}
\item[(i)] each left edge is glued to a unique right edge and vice versa,
\item[(ii)] each lower edge is glued to a unique upper edge and vice versa,
\item[(iii)] after the gluing process of all the sides we obtain a connected surface $Y$.
\end{enumerate}

Let $E_{0}$ be the torus obtained by gluing the left (respectively, the lower) side with the right (respectively, upper) side of the unit square. The above gluing process provides a (branched) covering map $\beta:Y \to E_{0}$, which can only have a branch value at the point of $E_{0}$ coming from the vertices of the unit square (the critical points are the corresponding points in $Y$ coming from the vertices of the used glued squares). We say that the above is a topological origami. If we provide of a Riemann surface structure $E$ to $E_{0}$, then we may lift it under $\beta$ to a Riemann surface structure $S$ on $Y$ such that $\beta:S \to E$ is an origami map. Similarly, origamis can be described by finite index subgroups of the free group of rank two. Details can be found, for instance, in \cite{HS,Lochak,S}.

Let $S$ be a closed Riemann surface of genus $g\geq 2$. It is kown that  $|{\rm Aut}(S)| \leq 84(g-1)$ (Hurwitz's upper bound \cite{Hurwitz}). Hurwitz's upper bound is attained for infinitely many values of $g$ (and also it is not for infinitely many others values of $g$) \cite{Macbeath}.
A group $H <{\rm Aut}(S)$  is called an origami group for $S$ if $S/H$ is a Riemann orbifold of genus one with exactly one cone point. In this case, as a consequence of the Riemann-Hurwitz formula, $|H| \leq 12(g-1)$ and that it can be generated by two non-commuting elements. If $(S,\beta)$ is a regular origami pair, then its deck group ${\rm Aut}(S,\beta)$ is an origami group.

In \cite{Hurwitz}, Hurwitz showed that every finite group $G$ can be realized as a group of conformal automorphisms of some closed Riemann surface of genus $g \geq 2$. In \cite[Thm. 4]{Greenberg2}, Greenberg observed that the Riemann surface can be chosen to have $G$ as its full group of conformal automorphisms. As a free group of rank at least two can be seen as a finite index subgroup of the free group of rank two, it follows that $G$ can be also realized as a subgroup of ${\rm Aut}(S,\beta)$ for a suitable origami pair $(S,\beta)$ of genus $g \geq 2$ (in that case, the Riemann orbifold $S/G$ necessarily has genus $\gamma \geq 1$ and, moreover, if $\gamma=1$, then $S/G$ has exactly one cone point). 

For $j=1,2$, let $X_{j}$ be a closed orientable surface and let $G_{j}$ be a finite group 
of orientation-preserving self-homeomorphisms of $X_{j}$. We assume $G_{1} \cong G_{2}$ (isomorphic groups). 
We say that the actions of $G_{1}$ and $G_{2}$ are {\it topologically equivalent} if there is an orientation-preserving homeomorphism $F:X_{1} \to X_{2}$ such that $F G_{1} F^{-1}=G_{2}$. As a consequence of the uniformization theorem, every finite group of orientation-preserving self-homeomorphisms of a closed orientable surface $X$ is topologically equivalent to the action of an isomorphic group of conformal automorphisms of a closed Riemann surface (see also \cite{Kerckhoff}). 

Let us assume that we realize $G$ as a group of conformal automorphisms of a closed Riemann surface $X$, of genus $g \geq 2$, and let us assume the Riemann orbifold $X/G$ has genus $\gamma \geq 1$. If $\gamma=1$ and $X/G$ has 
 exactly one cone point, then $G={\rm Aut}(X,\beta)$ for the origami pair $(X,\beta)$, where $\beta:X \to X/G$ is a regular branched coevering with ${\rm deck}(\beta)=G$. The following result takes care of the case when $X/G$ has genus at least two.

\begin{theo}\label{main1}
Let $G$ be a finite group of conformal automorphisms of  a closed Riemann surface $X$ of genus $g \geq 2$ such that $X/G$ is an orbifold of genus $\gamma \geq 2$.
Then there is an origami pair $(S,\beta)$, where $S$ has genus $g$, $G \cong {\rm Aut}(S,\beta)$ and the actions of ${\rm Aut}(S,\beta)$ and $G$ are topologically equivalent. 
\end{theo}

As every finite group acts on a Riemann surface as a group of conformal automorphisms with quotient Riemann orbifold of genus at least two, Theorem \ref{main1} asserts the following.

\begin{coro}\label{main0}
Every finite group is isomorphic to ${\rm Aut}(S,\beta)$ for a suitable origami pair $(S,\beta)$ of genus at least two. 
\end{coro}

We may define the origami genus of the finite group $G$ as the lowest genus $og(G)\geq 2$ such that $G$ is embedded in ${\rm Aut}(S,\beta)$ for a suitable origami $(S,\beta)$ with $S$ of genus $og(G)$. As a consequence of Theorem \ref{main1} we have the following simple fact.

\begin{coro}\label{coro2}
Let $G$ be a finite group and let $\sigma^{*}(G)$ be the minimal genus for a conformal action of $G$ as a group of conformal automorphisms of a surface of genus at least two and quotient orbifold also of genus at least two. Then $\sigma^{*}(G)=og(G)$.
\end{coro}

\begin{rema}
A version of Theorem \ref{main1} at the level of dessins d'enfants was previously obtained in \cite{H:dessins}. At this level, Corollary \ref{coro2} asserts that the strong symmetric genus of a finite group is also the its minimal genus action on dessins d'enfants.
\end{rema}


\section{Proof of theorem \ref{main1}}
Let us start with the following helpful observation.

\begin{lemm}\label{lema1}
For each pair of integers $\gamma \geq 2$ and $r \geq 0$, there is a non-uniform origami pair $(R,\delta)$, where $R$ is closed Riemann surface of genus $\gamma$, $\delta$ has some prime degree and whose fiber over its branch value of cardinality at least $r$.
\end{lemm}
\begin{proof}
Let us consider a topological origami defined by the gluing of $p=3\gamma+l-3$ squares, as shown in Figure \ref{figura1}, where $l \geq 1$ (in that figure, each lower side $j$ is glued to an upper side $j$, and each left side $a_{j}$ is glued to a right side $a_{j}$). The gluing procees produces a topological origami pair $(Y, \delta:Y \to E_{0}=S^{1}\times S^{1})$, where $Y$ has genus $\gamma$ and the primage of its unique branch value has cardinality $s=3\gamma+l-3$.
The six vertices of each of the $\gamma-1$ blocks formed of two squares (one above the other) produces a point on $Y$ being a critical point of degree $3$. Each of the two right vertices of the $l$ squares at the right part produces a  point on $Y$ which is not a critical point (the two left vertices of the first square at the left side of the figure are equivalent to the ones produced by the last square at the right). It follows that this is a non-uniform origami. So, we only need to assume $l \geq 1$ such that $p$ is a prime integer and $s \geq r$. Now, take a Riemann surface structure on $E_{0}$ and lift it under $\delta$ to obtain a Riemann surface structure $R$ on $Y$ which makes $\delta$ an origami map and $(R,\delta)$ an origami pair as needed.  
\end{proof}

\begin{figure}[htb]
\centering
\includegraphics[width=\textwidth]{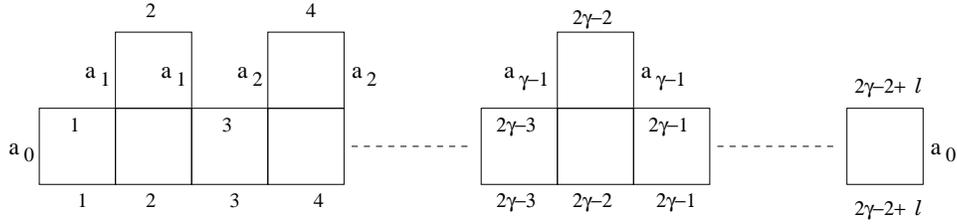}
\caption{An origami of genus $\gamma \geq 2$ with $3\gamma+l-3$ squares and $\gamma+l-1$ vertices}
\label{figura1}
\end{figure}

\subsection{Proof of theorem \ref{main1}}
Let $G$ be some finite group acting as a group of conformal automorphisms of a closed Riemann $X$ of genus $g \geq 2$ such that  the quotient orbifold $X/G$ consists of a closed Riemann surface of genus $\gamma \geq 1$ with some finite number of cone points, $p_{1},\ldots, p_{r}$, such that, for $r>0$, the cone point $p_{j}$ has cone order $m_{j} \geq 2$, and these numbers satisfy $2\gamma+r-2 > \sum_{j=1}^{r}m_{j}^{-1}$. 
As a consequence of the uniformization theorem, there are a Fuchsian group $K$ acting on the hyperbolic plane ${\mathbb H}^{2}$ with presentation
\begin{equation}
K=\langle A_{1},B_{1},\ldots, A_{\gamma},B_{\gamma}, C_{1},\ldots, C_{r}: \prod_{j=1}^{\gamma}[A_{j},B_{j}] \prod_{k=1}^{r} C_{k}=1=C_{1}^{m_{1}}=\cdots=C_{r}^{m_{r}}\rangle,
\end{equation}
where $[A_{j},B_{j}]=A_{j}B_{j}A_{j}^{-1}B_{j}^{-1}$,
and a surjective homomorphism $\theta:K \to G$, whose kernel $\Gamma_{\theta}$ is torsion-free, such that $X={\mathbb H}^{2}/\Gamma_{\theta}$, $X/G={\mathbb H}^{2}/K$ and the regular branched cover $X \to X/G$ is induced by the inclusion $\Gamma_{\theta} \leq K$.

Let us choose an origami pair $(R,\delta:R \to E)$, where $R$ has genus $\gamma$, whose $\delta$-preimage of its branch value $q \in E$ has cardinality at least $r$  and $\delta$ has prime degree (as in Lemma \ref{lema1}). Let us make a choice of points $q_{1},\ldots,q_{r} \in \delta^{-1}(q)$. As a consequence of quasiconformal deformation theory \cite{Nag}, we may find a Fuchsian group $\widehat{K}$ such that ${\mathbb H}^{2}/\widehat{K}$ is the orbifold whose underlying Riemann surface is $R$ and whose cone points are $q_{1},\ldots,q_{r}$ such that $q_{j}$ has cone order $m_{j}$. The group $\widehat{K}$ is isomorphic to $K$ (as abstract groups). Then there is a quasiconformal homeomorphism $W:{\mathbb H}^{2} \to {\mathbb H}^{2}$ conjugating $K$ into $\widehat{K}$. In this case, $S={\mathbb H}^{2}/W \Gamma_{\theta} W^{-1}$ is a closed Riemann surface admitting the group $\widehat{G}=\widehat{K}/W \Gamma_{\theta} W^{-1} \cong G$ as a group of conformal automorphisms and such that the actions of $G$ on $X$ and that of $\widehat{G}$ on $S$ are topologically equivalent (in particular, $S/\widehat{G}$ is also topologically equivalent to $S/G$). 

Consider the origami pair $(S,\beta=\delta \circ \eta)$, where $\eta:S \to R$ is a regular branched cover with deck group $\widehat{G}$.  In this case, the group $\widehat{G}$ is a subgroup of ${\rm Aut}(S,\beta)$ whose conformal action is topologically equivalent to that of $G$. We claim that $\widehat{G} = {\rm Aut}(S,\beta)$ as desired. In fact, let us assume, by 
 the contrary, that $\widehat{G}$ is a proper subgroup of  ${\rm Aut}(S,\beta=\delta \circ \eta)=\widetilde{G}$. It means that the branched cover $\delta:R \to \widehat{\mathbb C}$ factors through $S/\widetilde{G}$. As $\delta$ has prime degree and $\widehat{G} \neq \widetilde{G}$, it must be that $(S,\beta)$ is a regular origami. But as $\eta$ is an unbranched cover and $(R,\delta)$ is non-uniform, neither can be $(S,\beta)$, a contradiction as a regular origami pairs are uniform.



\begin{thebibliography}{99}


\bibitem{Greenberg2}
L. Greenberg.
Maximal Fuchsian groups.
{\it Bull. Amer. Math. Soc.} {\bf  69} (1963), 569--573.

\bibitem{HS}
F. Herrlich and G. Schmith\"usen.
Dessins d’enfants and origami curves.
In, Handbook of Teichmüller Theory, Volume II.
Editor: Athanase Papadopoulos (IRMA, Strasbourg, France) 2009. 
 pp: 767--809.
DOI: 10.4171/055-1/19

\bibitem{H:dessins}
R. A. Hidalgo.
Automorphisms groups of dessins d’enfants. 
{\it Archiv der Mathematik} {\bf 112} (2019), 13--18.

\bibitem{Hurwitz}
A. Hurwitz. 
\"Uber algebraische Gebilde mit eindeutigen Transformationen in sich. 
{\it Math. Annalen} {\bf 41} (1893), 403--442.


\bibitem{Kerckhoff}
S. Kerckhoff.
The Nielsen realization problem.
{\it Annals of Mathematics}, Second Series, {\bf 117}, No. 2 (1983), 235--265.

\bibitem{Lochak}
P. Lochak. 
On arithmetic curves in the moduli spaces of curves. 
{\it J. Inst. Math. Jussieu} {\bf 4} No.3 (2005), 443--508.

\bibitem{Macbeath}
A. M. Macbeath.
On a Theorem of Hurwitz.
{\it Glasgow Math. J.} {\bf 5} N. 2 (1961), 90--96.



\bibitem{Nag}
S. Nag.
{\it The complex analytic theory of Teichm\"uller spaces.}
A Wiley-Interscience Publication. John Wiley \& Sons, Inc., New York 1988.


\bibitem{S}
G. Schmith\"usen.
Veech Groups of Origamis. 
Ph. D. thesis, Universit\"at Karlsruhe, 2005.

\bibitem{Wink}
J. Winkelmann.
Countable Groups As Automorphism Groups.
{\it Documenta Mathematica} {\bf 7} (2002), 413--417.

\end{thebibliography}
\end{document}